\documentclass{amsart}
 \newtheorem{thm}{Theorem}[section]
 \newtheorem{cor}[thm]{Corollary}
 \newtheorem{lem}[thm]{Lemma}
 
 \theoremstyle{definition}
 
 \theoremstyle{remark}

 \numberwithin{equation}{section}

\newcommand{\abs}[1]{\left\vert#1\right\vert}

\newcommand{\G}{\mathsf{\Gamma}}
\newcommand{\B}{\mathsf{B}}

\newcommand{\rl}{{\mathbb{R}}}

\newcommand{\tmop}[1]{\ensuremath{\operatorname{#1}}}
\renewcommand{\Re}{\tmop{Re}}

\DeclareMathOperator{\sech}{sech}

\begin{document}

%
%
%
%
%
%
%
%
%

\title[A remarkable formula of Ramanujan]
 {On a remarkable formula of Ramanujan}

\author[D. Chakrabarti]{Debraj Chakrabarti}

\address{TIFR Centre for Applicable Mathematics, Sharada Nagar, Chikkabommasandra, Bengaluru- 560065, India.}

\email{debraj@math.tifrbng.res.in}

\author[G. K.  Srinivasan]{Gopala Krishna Srinivasan}
\address{Department of Mathematics, Indian Institute of Technology Bombay, Powai,  Mumbai-400076, India}
\email{gopal@math.iitb.ac.in}



\begin{abstract}
A simple proof of Ramanujan's formula for the Fourier transform of 
$\abs{\G(a + it)}^2$ is given where $a$ is fixed and has positive real part
 and $t$ is real. The result is extended to other values of $a$ by solving an inhomogeneous ODE and 
we use it to calculate the jump across the imaginary axis. 
\end{abstract}

\maketitle

\section{Introduction}

The behavior of the modulus of the gamma function along vertical lines plays an important role in analytic 
number theory (\cite[pp 767-775]{landau}.) Whenever the real number $a$ is not a negative integer or zero, the function 
\[
t \mapsto \abs{\G(a + it)}^2\] 
is an even function in the  Schwartz class, i.e., it decays rapidly at infinity along with all derivatives.  When $a>0$, one can prove
the following estimate due to  Lerch (\cite[p. 15]{godefroy}): 
\[
\abs{\G(a + it)} = \frac{\lambda\G(1+a)}{\sqrt{a^2 + t^2}}\sqrt{\frac{t}{\sinh\pi t}}, 
\]
where $\lambda$ satisfies $1 < \lambda < \sqrt{1+t^2}$, and if $a=\frac{1}{2}$, we have the closed form expression
\[ \abs{\G\left(\frac{1}{2}+it\right)}^2 = \frac{\pi}{\cosh \pi t}.\]

It is natural to ask what we can say about the Fourier transform of the Schwartz class function $t\mapsto \abs{\G(a + it)}^2$.
The following beautiful formula was proved by Srinivasa Ramanujan in \cite{ramanujan1}. For $a>0$, we have
\begin{equation}\label{eq-1p1}
\int_{-\infty}^{\infty} \abs{\G(a + it)}^2 \exp(-i\xi t) dt = 
\sqrt{\pi}\G(a)\G\left(a + \frac{1}{2}\right)\left(\mbox{cosh}\frac{\xi}{2}\right)^{-2a}.
\end{equation}
This  gives the Fourier transform of the rapidly decreasing (i.e., Schwartz class)  function $t\mapsto \abs{\G(a + it)}^2$, and by Fourier Inversion,
also the Fourier transform of $\xi\mapsto \left(\mbox{cosh}\frac{\xi}{2}\right)^{-2a}$. The existence of 
either of these transforms as a closed form expression is a minor miracle. Setting $a=\frac{1}{2}$ in \eqref{eq-1p1} gives
the well-known formula
\[ \int_{-\infty}^\infty\frac{1}{\cosh\pi t} e^{-i\xi t} dt = \frac{1}{\cosh\frac{\xi}{2}},\]
revealing the close relation with the Gaussian function.

The integral \eqref{eq-1p1} is the simplest in a class of integrals considered by 
Mellin, Ramanujan, Hardy,  Hecke and others (\cite[p.~98, pp.~103-104, and p.~492]{hardy}). Mellin studied 
integrals of this form
 in connection with representations of solutions of first order difference equations (\cite[pp. 344-345]{mellin1}) and 
solutions of (generalized) hyper-geometric equations (\cite[ pp 148-149]{mellin2}).  Integrals 
of this form also play an important role in analytic number theory (\cite[ p. 152]{hardy-littlewood}) 
and in the theory of real quadratic number 
fields (\cite[p.~58]{rademacher}). 

In this note we consider 
what happens to the formula \eqref{eq-1p1} when $-1< \Re (a) <0$.  It is clear that the formula as it stands must 
fail when $\Re (a)$ is negative and $a \notin \mathbb Z$ 
since for such $a$, the function  $t\mapsto |\G(a+it)|^2$ is rapidly decreasing, but 
$\xi\mapsto\cosh^{-2a}\xi$ on the right hand side blows up exponentially. 
We obtain, using Fourier analysis, an inhomogeneous ordinary differential equation for the integral when 
$-1 < \Re (a) < 0$. The formula assumes a very different form when $\Re (a) < 0$. 
To make the paper self-contained we offer using techniques of Fourier analysis, an alternate proof of 
 \eqref{eq-1p1} 
when $\Re (a)>0$.

We will prove Ramanujan\rq{}s formula in the following  form,
where we allow $a$ to be complex:
\begin{thm} \label{thm-main} When $\Re (a) > 0$, 
\begin{equation}\label{eq-2p2}
\int_{-\infty}^{\infty} \G(a + it)\G(a-it) \exp(-i\xi t) dt = \frac{2\pi\G(2a)}{4^a\cosh^{2a}\left(\frac{\xi}{2}\right)}. \end{equation}
\end{thm}
When $a$ is real,  the left hand side of \eqref{eq-2p2} reduces to that of \eqref{eq-1p1}, 
since $\overline{\G(z)} = \G({\overline{z}})$ for all $z$ in the domain of 
$\G$, and therefore
\[ \G(a + it)\G(a - it)=\abs{\G(a+it)}^2  ,
\]
and in the right hand side, the two expressions coincide, thanks to the classical duplication 
formula for the gamma function  (\cite[p. 35]{remmert}, \cite[formula~3.8]{srinivasan}):
\begin{equation}\label{dup}
\sqrt{\pi}\G(2a) = 2^{2a-1}\G(a)\G\left(a+\frac{1}{2}\right).
\end{equation}

If we try to prove Theorem~\ref{thm-main} by formally computing the Fourier transform, we soon
run into a familiar obstacle, namely, the appearance of an oscillatory integral  
\[
\int_{-\infty}^{\infty} \exp(i\xi(x-y)) dx.
\]
The standard way to cope with this (see \cite[p. 37]{strichartz})
is to introduce a Gaussian convergence factor $\exp(-\epsilon x^2)$ (with $\epsilon>0$) 
in the integrand, compute the integral, and let 
let $\epsilon \rightarrow 0^+$.  This is done in Section~\ref{sec-main} below.

We then ask how to modify \eqref{eq-2p2} when $\Re(a)$ is negative.  We show that in  Section~\ref{sec-jump}
the question can be reduced to the solution of an inhomogeneous linear  ODE, whose solution 
can be reduced to (repeated) quadratures.  When $-1<\Re(a)<0$ we prove the following:
\begin{thm} \label{thm-negative}When $\Re(a) \in (-1, 0)$,  we have
\begin{align} \int_{-\infty}^{\infty} \G(a + it)\G(a-it) \exp(-i\xi t) dt &=&\nonumber\\
-\sqrt{\pi}\G(a+1)\G\left(a+\frac{3}{2}\right)
\left\{
\cosh(a\xi)\int_0^{\infty}\frac{e^{as}}{\cosh^{2a+2}\left(\frac{s}{2}\right)}ds\right. + 
&\left.\int_0^{\xi}\frac{\cosh\left(a(\xi - s)\right)}{\cosh^{2a+2}\left(\frac{s}{2}\right)} ds\right\}\label{eq-extension}
\end{align}
\end{thm} 
As a consequence, we find the jump in the Fourier transform of $t\mapsto \G(a-it)\G(a+it)$ as\
$a$ crosses the imaginary axis in Corollary~\ref{cor-jump}.

Theorem~\ref{thm-negative} shows that when $a<0$, the Fourier transform is no longer in closed form, but can 
still be expressed in terms of quadratures. It will be seen in the proof that this method can be extended to obtain 
similar formulas for vertical strips of the form $-(n+1)<\Re(a)<-n$, for any positive integer $n$ in terms of multiple 
quadratures.

\section{ Proof of Theorem~\ref{thm-main}} \label{sec-main}

We begin with a useful formula that goes back to 
Binet (\cite[p.136]{binet}).  We recall  that the beta function $\B(p, q)$ is defined for 
$\Re(p)>0$ and $\Re(q)>0$ to be 
\[
\B(p, q) = \int_0^1 s^{p-1}(1-s)^{q-1} ds,
\]
and the basic relation between the beta and gamma functions (\cite[p. 35]{remmert}, \cite[p. 313]{srinivasan}):
\begin{equation}\label{beta-gamma}
\G(p)\G(q) = \G(p+q)\B(p, q).
\end{equation}
\begin{lem}[Binet] \label{lem-binet}When $\Re(p) > 0$ and  $\Re(q) > 0$, 
\begin{equation}\label{eq-2p1}
\B(p, q) = \int_{-\infty}^{\infty}\frac{e^{(p-q)u} + e^{(q-p)u}}{(e^u + e^{-u})^{p+q}}  du.
\end{equation}\end{lem}
\begin{proof} We set $s = (1+e^{2u})^{-1}$ in the integral defining $\B(p, q)$: 
\begin{align*}
\B(p, q)& =  2\int_{-\infty}^{\infty} \frac{e^{2uq}}{(1 + e^{2u})^{p+q}}du\\ 
&= 
\int_{-\infty}^{\infty} \frac{e^{2up} + e^{2uq}}{(1+e^{2u})^{p+q}}du,
\end{align*}
since the beta function is symmetric in $p$ and $q$. Multiplying the 
numerator and denominator of the integrand in the  last integral by $\exp\left(-u(p+q)\right)$, we get 
\eqref{eq-2p1}. 
\end{proof}

We now prove Theorem~\ref{thm-main}. By virtue of the beta-gamma relation \eqref{beta-gamma}:
\[
\G(a + it)\G(a - it) = \B(a + it, a -it)\G(2a).
\]
Invoking Lemma~\ref{lem-binet},  we  see that the left hand side 
\[ \int_{-\infty}^{\infty} \G(a + it)\G(a-it) \exp(-i\xi t) dt\]
of \eqref{eq-2p2} is equal to 
\begin{align}
 &\G(2a)\int_{-\infty}^{\infty}\exp(-it\xi)\left(\int_{-\infty}^{\infty}
\frac{(e^{2itu} + e^{-2itu})}{(e^u + e^{-u})^{2a}}du\right)dt\nonumber\\
&= 
\G(2a)\int_{-\infty}^{\infty}\lim_{\epsilon\rightarrow 0^+}\left\{\exp(-\epsilon t^2 - it\xi)\left(\int_{-\infty}^{\infty}
\frac{(e^{2itu} + e^{-2itu})}{(e^u + e^{-u})^{2a}}du\right)\right\}dt,\label{eq-dctjust1}
\end{align}
Where in the last line, we have inserted a convergence factor $\exp(-\epsilon t^2)$  (with $\epsilon>0$).
Note that the function
\[
t \mapsto \int_{-\infty}^{\infty} \frac{e^{2itu} + e^{-2itu}}{(e^u + e^{-u})^{2a}} \;du
\]
is rapidly decreasing, since it can be interpreted as the sum of Fourier transforms of 
rapidly decreasing functions, and therefore is in particular in $L^1$. Therefore, we are justified in 
applying the dominated convergence theorem, and writing \eqref{eq-dctjust1} as
\begin{align*}
 &\G(2a)\;\lim_{\epsilon\rightarrow 0}
\int_{-\infty}^{\infty}\exp(-\epsilon t^2 - it\xi) \left(
\int_{-\infty}^{\infty}\frac{e^{2itu} + e^{-2itu}}{(e^u + e^{-u})^{2a}}du \right)dt \\
&= \G(2a) \;\lim_{\epsilon\rightarrow 0}
\int_{-\infty}^{\infty}\frac{1}{(e^u + e^{-u})^{2a}}\left(
\int_{-\infty}^{\infty}\left(e^{-it(-2u+\xi)-\epsilon t^2}+e^{-it(2u+\xi)-\epsilon t^2}\right)dt\right) du,\\
\end{align*}
where in the last line the use of Fubini\rq{}s theorem to interchange the order of integration is easily justified, since $t\mapsto \exp(-\epsilon t^2)$
is integrable.  Inserting the well-known formula (\cite[p.~41]{strichartz}) 

\[ 
\int_{-\infty}^{\infty} \exp(-\epsilon t^2 - ict)dt = \sqrt{\frac{\pi}{\epsilon}}\exp\left(-\frac{c^2}{4\epsilon}\right)
\]
for the Fourier transform of the Gaussian into each term of the second integral factor,  we get
\begin{equation}\label{eq-2int}
\sqrt{\pi}\G(2a)\;\lim_{\epsilon\rightarrow 0}\left\{ 
\int_{-\infty}^{\infty}\frac{e^\frac{{-(2u-\xi)^2}}{4\epsilon}}{\sqrt{\epsilon}}\cdot\frac{1}{(e^u + e^{-u})^{2a}}du + 
\int_{-\infty}^{\infty}\frac
{e^\frac{-(2u+\xi)^2}{4\epsilon}}{\sqrt{\epsilon}}\cdot\frac{1}{(e^u + e^{-u})^{2a}}du
\right\}
\end{equation}
The change of variables $u = \frac{\xi}{2} + \sqrt{\epsilon}\;w$ transforms the first integral in \eqref{eq-2int} into 
\begin{equation}\label{eq-3int}
\int_{-\infty}^{\infty} 
\frac{e^{-w^2}dw}{(e^{\frac{\xi}{2} + \sqrt{\epsilon}\;w} + e^{-\frac{\xi}{2} - \sqrt{\epsilon}\;w})^{2a}}.
\end{equation}
Since the denominator is always greater than or equal to 1, in this integral,  the integrand 
is dominated by the $L^1$ function $w\mapsto e^{-w^2}$.  We may therefore  apply the 
dominated convergence theorem and conclude that the integral \eqref{eq-3int}
converges as $\epsilon\rightarrow 0^+$ to 
\begin{equation}\label{eq-limit}\sqrt{\pi}\cdot \frac{1}{\left(e^{\frac{\xi}{2}} + e^{-\frac{\xi}{2}}\right)^{2a}} =\frac{\sqrt{\pi}}{{4^a}\cosh^{2a}\left(\frac{\xi}{2}\right)}.\end{equation}

The change of variables $v= -u$ shows that the second integral in \eqref{eq-2int} is equal to the first for each $\epsilon>0$, and therefore,
its limit as $\epsilon\to 0^+$ is also equal \eqref{eq-limit}, and therefore, 
the upshot of our computation is that the expression in \eqref{eq-2int}, and therefore, the Fourier transform on the left hand side of
\eqref{eq-2p2} is equal to
\[ \frac{2\pi\G(2a)}{4^a\cosh^{2a}\left(\frac{\xi}{2}\right)} ,\]
which is the required right hand side.

{\bf Remark:} That the first integral in \eqref{eq-2int} converges as $\epsilon\to 0^+$ to the quantity in \eqref{eq-limit}
may also be deduced from the following fact: for a fixed $\xi$, as $\epsilon\to 0^+$, the first factor 
$u\mapsto \epsilon^{-\frac{1}{2}} e^{\frac{{-(2u-\xi)^2}}{4\epsilon}}$ of the first integral in \eqref{eq-2int} converges  in the sense 
of distributions to $\sqrt{\pi}\delta_{\frac{\xi}{2}}$, where  for $p\in\rl$,  we denote by $\delta_p$ the delta distribution on $\rl$ supported 
at the point $p$.


\section{ Proof of Theorem~\ref{thm-negative}} \label{sec-jump}

Suppose that $\Re(a)$ is not a negative integer or 0. For real $\xi$, we define 
\begin{equation}\label{fourier}
I(a,\xi) =  \int_{-\infty}^{\infty} a \G(a-it)\G(a+it) e^{-it\xi} dt.
\end{equation}
This is well-defined, since the integrand is easily seen to be a rapidly decreasing function of $t$.
When $\Re(a)>0$, Theorem~\ref{thm-main} along with the duplication formula \eqref{dup} shows that  
\begin{equation}\label{eq-mainbis} 
I(a,\xi) = a\Phi(a) \sech^{2a}\left(\frac{\xi}{2}\right),
\end{equation}
where we denote 
\[ \Phi(a)= \sqrt{\pi}\G(a)\G\left(a+\frac{1}{2}\right),\]
and  $\sech(t)= \left(\cosh(t)\right)^{-1}$. The basic property of the function $I(a,\xi)$ is given in the following lemma:
\begin{lem}
If $\Re(a)$ is not a negative integer or zero, we have
\begin{equation}\label{eq-ode-iaxi}
\left(\frac{d^2}{d\xi^2} - a^2\right)I(a, \xi)= \left(\frac{-a}{a+1}\right)I(a+1,\xi).\end{equation}
\end{lem}
\begin{proof}
Using the functional relation $\G(z)= \displaystyle{\frac{\G(z+1)}{z}}$ twice, the right hand side of 
 \eqref{fourier} can be rewritten as 
\begin{equation}\label{fourier1}
I(a,\xi) =  \int_{-\infty}^{\infty} \left(\frac{a}{a^2 + t^2}\right) \G(1+a-it)\G(1+a+it) e^{-it\xi} dt.
\end{equation}
To get rid of the factor $a^2 + t^2$ in the denominator in the integrand in \eqref{fourier1} we 
differentiate under the integral twice with respect to  $\xi$, and subtract $a^2 I(a,\xi)$.  We obtain 
\begin{align*}
\left(\frac{d^2}{d\xi^2} - a^2\right)I(a, \xi)& = \int_{-\infty}^{\infty} \left(\frac{-at^2}{a^2 + t^2}\right) \G(1+a-it)\G(1+a+it) e^{-it\xi} dt\\
&\phantom{a^2+t^2}-\int_{-\infty}^{\infty} \left(\frac{a^3}{a^2 + t^2}\right) \G(1+a-it)\G(1+a+it) e^{-it\xi} dt\\
&=-a \int_{-\infty}^{\infty}\G(1+a-it)\G(1+a+it) e^{-it\xi} dt \\
&=\left(\frac{-a}{a+1}\right)\int_{-\infty}^{\infty}(a+1)\G(1+a-it)\G(1+a+it) e^{-it\xi} dt\\
&= \left(\frac{-a}{a+1}\right)I(a+1,\xi).
\end{align*}
\end{proof}
We now consider the situation in which $\Re(a)>-1$, and $\Re(a)\not=0$. Then $\Re(a+1)>0$, and consequently,
putting the value of $I(a+1,\xi)$ from  \eqref{eq-mainbis} into \eqref{eq-ode-iaxi} we obtain 
that
\[ \left(\frac{d^2}{d\xi^2} - a^2\right)I(a, \xi)= 
\left(\frac{-a}{a+1}\right) \cdot(a+1)\Phi(a+1)\sech^{2(a+1)}\left(\frac{\xi}{2}\right).\]
We have thus obtained the following inhomogeneous linear differential equation 
satisfied by the function $\xi\mapsto I(a,\xi)$:
\begin{equation}\label{ode}
\left(\frac{d^2}{d\xi^2} - a^2\right)I(a, \xi) = 
-a\Phi(a+1){\sech}^{2a+2}\frac{\xi}{2}. 
\end{equation}
The general solution of the associated homogeneous equation \[ \left(\frac{d^2}{d\xi^2} - a^2\right)I(a, \xi) = 0\]
 is $\xi\mapsto Ae^{a\xi} + Be^{-a\xi}$, where $A$ and $B$ 
are constants. 
We proceed to find a particular solution of \eqref{ode} by the standard 
method of variation of parameters. We seek the particular integral in the form 
\begin{equation}\label{part-int}
\xi\mapsto v_1(\xi)e^{a\xi} + v_2(\xi)e^{-a\xi}, 
\end{equation}
where the functions $v_1$ and $v_2$ are determined by the system  of equations 
\begin{align*}
v_1^{\prime}(\xi)e^{a\xi} + v_2^{\prime}(\xi)e^{-a\xi} =  &  0, \\ 
v_1^{\prime}(\xi)e^{a\xi} - v_2^{\prime}(\xi)e^{-a\xi} = & -\Phi(a+1)\mbox{sech}^{2a+2}\frac{\xi}{2}.\\
\end{align*}
Solving the above pair of equations for $v_1\rq{}$ and $v_2\rq{}$, representing $v_1$ and $v_2$ as 
indefinite integrals,  and  substituting in \eqref{part-int} gives the particular integral 
\begin{align*}
-\left(\frac{\Phi(a+1)}{2}\int_0^{\xi}{\sech}^{2a+2}\left(\frac{s}{2}\right)e^{-as}ds\right)\cdot e^{a\xi}  + \\
\left(\frac{\Phi(a+1)}{2}\int_0^{\xi}{\sech}^{2a+2}\left(\frac{s}{2}\right)e^{as}ds\right)\cdot e^{-a\xi}.
\end{align*}
The complete solution of \eqref{ode} is thus 
\begin{align}
I(a, \xi) &= e^{a\xi}\left(A - \frac{\Phi(a+1)}{2}\int_0^{\xi}\sech^{2a+2}\left(\frac{s}{2}\right)e^{-as}ds\right) +\nonumber\\ 
&\phantom{=}e^{-a\xi}\left(B + \frac{\Phi(a+1)}{2}\int_0^{\xi}\sech^{2a+2}\left(\frac{s}{2}\right)e^{as}ds\right). \label{ode-sol}
\end{align}The constants $A$ and $B$ are to be determined by side conditions such as initial or 
boundary conditions. 
The Riemann Lebesgue lemma (\cite[p. 106 ff]{strichartz}) which asserts  that  the Fourier transform of an integrable function decays at 
infinity, provides the requisite boundary conditions
\begin{equation}\label{riemann}
\lim_{\xi\rightarrow \infty}I(a, \xi) = 0,\quad \lim_{\xi\rightarrow -\infty}I(a, \xi) = 0 
\end{equation}
Taking the limits as $\xi\to\infty$ and $\xi\to -\infty$ in \eqref{ode-sol}, and 
using \eqref{riemann} we get the following values of $A$ and $B$, we see that if $\Re(a)<0$, we have
\begin{equation}\label{coeff}
A = B = -\frac{\Phi(a+1)}{2}\int_0^{\infty}{\sech}^{2a+2}\left(\frac{s}{2}\right)e^{as} ds.
\end{equation}
For use in the next section, we note that when $\Re(a)>0$, the same method gives us that
\begin{equation}\label{coeff2}
A = B = \frac{\Phi(a+1)}{2}\int_0^{\infty}\mbox{sech}^{2a+2}\left(\frac{s}{2}\right)e^{-as} ds. 
\end{equation}
Substitution of the values of $A$ and $B$ obtained in \eqref{coeff} in \eqref{ode-sol}, and simplifying the expression,
 gives us  that
for $-1<\Re(a)<0$, we have
\begin{align*}I(a, \xi) = -\Phi(a+1)\left\{
\int_0^{\xi}{\sech}^{2a+2}\left(\frac{s}{2}\right)\cosh\left(a(\xi - s)\right) ds +\right.\\ 
\left. \cosh a\xi\int_0^{\infty}\mbox{sech}^{2a+2}\left(\frac{s}{2}\right)e^{as}ds\right\},
\end{align*}
and this is equivalent to Theorem~\ref{thm-negative} in view of the definitions of $I(a,\xi)$ and $\Phi(a)$.

{\bf Remark:} Note that the method of proof can be used to generate  formulas for $I(a,\xi)$ (and consequently
the left hand side of \eqref{eq-extension}) for  $\Re(a)<-1$ by a recursive process. These formulas will involve repeated
quadratures.  For example, suppose  for a positive integer $n$  that $I(a,\xi)$ has been determined for $\Re(a)\in (-n,-n+1)$.
Then the right hand of \eqref{eq-ode-iaxi} is known when $\Re(a)\in (-n-1,n)$. Consequently, we have a 
linear inhomogeneous ordinary differential equation to find $I(a,\xi)$ when  $\Re(a)\in (-n-1,n)$,  and this
 can again be solved using the method of variation of parameters, yielding a solution in terms of multiple quadratures,
where the integrands now involve only elementary functions.
\section{The Jump Across the Imaginary Axis}
As we have seen in the previous section, for fixed real $\xi$, the left hand side of \eqref{eq-2p2} 
is a holomorphic function 
of $a$, provided $\Re(a)$ is not a negative integer or zero, and the technique of the last section may be used 
to obtain formulas representing this function in terms of integrals of elementary functions. Therefore we can find
the jump of this function of $a$ along the lines $\{\Re(a)=-n\}$, where $n$ is a non-negative integer. We now carry out this program
for the simplest case of $n=0$. 
Denote for convenience 
\[ J(a,\xi)=  \int_{-\infty}^{\infty}  \G(a-it)\G(a+it) e^{-it\xi} dt,\]
and for a point $p$ on the imaginary axis consider the jump\[
[[J(p, \xi)]] = \lim_{\substack{a\to p\\ \Re(a)>0}}J(a,\xi)-\lim_{\substack{a\to p\\ \Re(a)<0}}J(a,\xi).  \]

\begin{cor}\label{cor-jump} 
For each real $\xi$,  and for purely imaginary $p\not=0$, we have
\[
[[J(p, \xi)]] = 4\pi \cosh(p\xi) \G(2p).
\]
\end{cor}
\begin{proof}
Notice that, the function $J$ above and the function $I$ of \eqref{fourier} are related by
\[ aJ(a,\xi)= I(a,\xi).\]
Set $p= ic$, where $c\in \rl$, and let $[[I(ic,\xi)]]$ denote the jump across the imaginary axis:
\[ [[I(ic, \xi)]] = \lim_{\substack{a\to ic\\ \Re(a)>0}}I(a,\xi)-\lim_{\substack{a\to ic\\ \Re(a)<0}}I(a,\xi).\]
In order to prove the result, it suffices to show that  for every nonzero real $c$, we have
\[
[[I(ic, \xi)]] = 2\pi\cos c\xi\G(1+2ic).
\]

The integrals in \eqref{ode-sol} are continuous functions of $a$ since the real part of $a+1$ is positive. 
Hence contribution to the jump of $I(\cdot, \xi)$  at the point $ic$ on the imaginary axis   
arises from the discontinuity in the 
coefficients $A$ and $B$. From equations \eqref{coeff} and\eqref{coeff2}, the difference in the right and left hand limits of $Ae^{a\xi}$ is 
\begin{align*}
e^{ic\xi}&\Phi(1+ic)2^{2ic+1}
\int_0^{\infty}\frac{(e^{ics}+e^{-ics})ds}{(e^{s/2} + e^{-s/2})^{2ic+2}}\\& = 
e^{ic\xi}\Phi(1+ic)2^{2ic+1}
\int_{-\infty}^{\infty}\frac{(e^{2icu}+e^{-2icu})du}{(e^{u} + e^{-u})^{2ic+2}}.
\end{align*}
Using Binet's formula with $p = 2ic+1$ and $q = 1$ we simplify the last expression to 
\[
2^{1+2ic}\Phi(1+ic)e^{ic\xi}\B(1+2ic, 1) =  
2^{2ic}\sqrt{\pi}\G(1+ic)\G(ic + \frac{1}{2})e^{ic\xi}. 
\]
Using the duplication formula \eqref{dup} to simplify the last expression, 
we see that the difference in the right and left hand limits of 
$Ae^{a\xi}$ at $ic$ is 
\[
\pi e^{ic\xi} \G(1+2ic).
\]
Similarly, the difference in the right and left hand limits of $Be^{-a\xi}$ at $ic$ is 
\[
\pi e^{-ic\xi}\G(1+2ic).
\]
Thus we have 
\[
[[I(ic, \xi)]] =\pi(e^{ic\xi} + e^{-ic\xi})\G(1+2ic) = 2\pi\cos c\xi\G(1+2ic)
\]
as required. 
\end{proof}
\paragraph*{Acknowledgment} The authors  thank the referee for his helpful comments and suggestions.

\end{document}